\documentclass[a4paper,12pt]{amsart}

\newtheorem{theorem}{Theorem}[section]
\newtheorem{lemma}[theorem]{Lemma}
\newtheorem{cor}[theorem]{Corollary}

\theoremstyle{definition}
\newtheorem{definition}[theorem]{Definition}

\theoremstyle{remark}
\newtheorem{remark}[theorem]{Remark}

\numberwithin{equation}{section}

\newcommand{\C}{\mathbb{C}}

\newcommand{\R}{\mathbb{R}}

\DeclareMathOperator{\ind}{ind}

\DeclareMathOperator{\ch}{ch}
\DeclareMathOperator{\coker}{coker}

\title{Witten genera of complete intersections}

\author{Michael Wiemeler}
\address{Mathematisches Institut, Universit\"at M\"unster, Einsteinstrasse 62, 48149 M\"unster, Germany}
\email{wiemelerm@uni-muenster.de}
\thanks{The research for this paper was funded by the Deutsche Forschungsgemeinschaft (DFG, German Research Foundation) under Germany's Excellence Strategy EXC 2044 –390685587, Mathematics M\"unster: Dynamics–Geometry–Structure and through CRC1442 Geometry: Deformations and Rigidity at University of M\"unster.
}



\begin{document}
\begin{abstract}
  We prove vanishing results for Witten genera of string generalized complete intersections in homogeneous \(\text{Spin}^c\)-manifolds and in other \(\text{Spin}^c\)-manifolds with Lie group actions.
  By applying these results to Fano manifolds with second Betti number equal to one we get new evidence for a conjecture of Stolz.
\end{abstract}

\maketitle


\section{Introduction}

The Witten genus was introduced by Witten as an analogue of the index of a Dirac operator for the free loop space of a closed manifold.
It is of particular interest when the free loop space satisfies a spin condition.
This is the case precisely when the manifold has a string structure, i.e. it has a spin structure and \(\frac{p_1}{2}(M)=0\).
A orientable manifold has a spin structure if and only if \(w_2(M)=0\).
Here \(\frac{p_1}{2}(M)\) is a characteristic class which is defined for spin manifolds and refines the first Pontrjagin class \(p_1(M)\) of \(M\).
Moreover, \(w_2(M)\) is the second Stiefel Whitney class of \(M\).

Analogously to vanishing results for the index of a Dirac operator on a closed spin manifold there are conjectures for the vanishing of the Witten genus of a closed string manifold admitting a metric of positive Ricci curvature (by Stolz \cite{MR1380455}) or admitting a non-trivial \(S^1\)-action (see for example \cite{MR2581907}).

So far all the known vanishing results for the Witten genus fall into three different classes:

\begin{enumerate}
\item The first class is for string manifolds with smooth Lie group actions (see  \cite{MR1331972},  \cite{0963.19002}, \cite{0944.58019}, \cite{MR3599868}, \cite{wiemeler2024conjecture}).
\item The second class is for certain string generalized complete intersections in specific complex manifolds (see \cite[p. 87]{MR1189136}, \cite{foerster07:_stolz}, \cite{MR2373447}, \cite{MR3275823}, \cite{MR3546614}, \cite{xiao2017witten}, \cite{han2023iterated}).
  \item More recently there is also a vanishing result for Witten genera of Riemannian string manifolds satisfying certain curvature conditions (see \cite{bettiol2022curvatureoperatorsrationalcobordism}). However these conditions seem to be unrelated to the condition of positive Ricci curvature.
\end{enumerate}

A generalized complete intersection in a closed oriented manifold \(M\) is the iterated transversal intersection of closed oriented codimension two submanifolds of \(M\).
If \(X\subset M\) is a generalized complete intersection, then the normal bundle of \(X\) extends to a vector bundle over \(M\).
The study of generalized complete intersections and virtual manifolds goes back to Hirzebruch, Thom and others (see for example \cite[Chapter 3]{MR1189136}).

A \(\text{Spin}^c\)-manifold is an orientable manifold with a \(\text{Spin}^c\)-structure.
Such a structure exists on an orientable manifold \(M\) if and only if the second Stiefel Whitney class of \(M\) lies in the image of the natural map \(H^2(M;\mathbb{Z})\rightarrow H^2(M;\mathbb{Z}_2)\).
Examples of such manifolds are unitary manifolds, i.e. manifolds whose stable tangent bundles have complex structures.

Here we prove the following two theorems.
The first is a vanishing result for Witten genera of string generalized complete intersections in homogeneous \(\text{Spin}^c\) manifolds.

\begin{theorem}
  \label{sec:introduction-6}
  Let \(G\) be a compact simply connected Lie group and \(H\subset G\) a closed subgroup which is not a maximal torus of \(G\) such that \(M=G/H\) is \(\text{Spin}^c\) and \(G\) acts almost effectively on \(M\).

  Let \(X\subset M\) be a generalized complete intersection and \(V\rightarrow M\) the extended normal bundle of \(X\).
  Assume that the following holds:
  \begin{enumerate}
  \item \(w_2(V)=w_2(M)\) and
  \item \(p_1(V)=p_1(M)\) modulo torsion.
  \end{enumerate}
  Then the Witten genus of \(X\) vanishes.
\end{theorem}

For string generalized complete intersections in \(\text{Spin}^c\) manifolds with non-homoge\-neous actions we have more generally:

  \begin{theorem}
    \label{sec:introduction-5}
  Let \(M\) be a closed connected \(Spin^c\) manifold such that
  \begin{enumerate}
  \item There is an almost effective action of a simply connected compact Lie group \(G\) on \(M\) such that \(M^{G}\neq \emptyset\), or
  \item \(b_1(M)=0\) and there is an almost effective action of a torus \(T\) on \(M\) with \(\dim T>b_2(M)\).
  \end{enumerate}

  Let, moreover, \(X\subset M\) be a generalized complete intersection and \(V\rightarrow M\) be the extended normal bundle of \(X\).
  Assume that the following holds:
  \begin{enumerate}
  \item \(w_2(V)=w_2(M)\) and
  \item \(p_1(V)=p_1(M)\) modulo torsion.
  \end{enumerate}
  Then the Witten genus of \(X\) vanishes.
\end{theorem}

Here \(b_i(M)\) denotes the \(i\)-th Betti number of \(M\).

The condition on the Stiefel-Whitney classes in the theorems imply that \(X\) is spin, whereas the conditions on the Pontrjagin classes almost imply that \(X\) is string.

A Fano manifold is a compact connected K\"ahler manifold whose first Chern class can be represented by a positive \((1,1)\)-form  (see \cite{MR1668579}).
By Yau's solution \cite{MR0480350} of the Calabi conjecture, all Fano manifolds admit K\"ahler metrics of positive Ricci curvature.
Note that all Fano manifolds are simply connected.

As a consequence of Theorem~\ref{sec:introduction-5} we get the following corollary which gives some new evidence for the Stolz conjecture.

\begin{cor}
  \label{sec:introduction-7}
  Let \(M\) be a Fano manifold with \(b_2(M)=1\). Assume that there is a smooth effective action of a torus \(T\) with \(\dim T=2\) on \(M\) or that there is a smooth almost effective action of a simply connected compact Lie group on \(M\) with a fixed point.
  
  Let, moreover, \(X\subset M\) be a string complete intersection of complex dimension at least four.

  Then \(X\) is Fano and the Witten genus of \(X\) vanishes.
\end{cor}

In the situation of the corollary by a complete intersection in \(M\) we mean a transverse intersection of several complex hypersurfaces of \(M\).
Note that for a string manifold \(Y\) of positive real dimension less than eight there is an integer \(k>0\) such that \(kY\) bounds an orientable manifold.
Therefore if, in the situation of Corollary~\ref{sec:introduction-7},  one has \(0<\dim_{\mathbb{C}} X<4\), then the Witten genus of \(X\) vanishes.

To the best of the author's knowledge, Theorem~\ref{sec:introduction-6} generalizes all previously known vanishing results for Witten genera of string generalized complete intersections in homogeneous spaces  (\cite[p. 87]{MR1189136}, \cite{foerster07:_stolz}, \cite{MR2373447}, \cite{MR3275823}, \cite{MR3546614}).
These previously known results are for very specific situations in which either there is a group action on \(X\) or there is so much information about the cohomology of \(X\) and \(M\) and the characteristic classes of these manifolds known that the Witten genus of \(X\) can be computed directly.
The advantage of our proof is that it needs neither an action on \(X\) nor any information about the cohomology ring structure of \(M\).
What suffices is the group action on \(M\).

As a consequence of Theorem~\ref{sec:introduction-5} we get a generalization of a recent result from \cite{han2023iterated} (see Corollary~\ref{sec:proofs-main-results-2}).
This corollary gives a vanishing result for Witten genera for certain string generalized complete intersections in so-called generalized Bott manifolds which are not Bott manifolds (see Section~\ref{sec:consequences-1} for the definition of these types of manifolds).

Theorem~\ref{sec:introduction-5} also partially generalizes the main result of \cite{xiao2017witten} which gave a vanishing result for certain string generalized complete intersections in toric manifolds.
Our condition on the characteristic classes of \(V\) is weaker than the corresponding condition in that result.
However, while Xiao's result works for arbitrary toric manifolds \(M\), in our case there is the condition \(2b_2(M)<2\dim T=\dim M\) where \(T\) is the torus acting on \(M\).

The proofs of the theorems are based on the observation that the Witten genus of \(X\) is equal to the index of some twisted Dirac operator on \(M\).
The theorems then follow from vanishing results for these indices which were proved in \cite{MR3031643} and \cite{MR3599868}.

The claim about the Witten genus of \(X\) in the corollary follows from Theorem~\ref{sec:introduction-5} in combination with the Lefschetz hyperplane theorem (see \cite{MR0215323}).
The claim about the Ricci curvature of \(X\) in the corollary follows from an application of an inequaltity for the second Chern class of a Fano manifold from \cite{MR4000253}.

This paper is organized as follows in Sections \ref{intro} and \ref{sec:textspinc-manifolds} we collect preliminaries on (generalized) complete intersections and \(\text{Spin}^c\)-manifolds.
Then in Section~\ref{sec:proofs-main-results-4} we prove the two theorems from above.
In Section~\ref{sec:consequences-1} we discuss Corollary~\ref{sec:introduction-7} and other consequences of these theorems.

I would like to thank Hans-Joachim Hein for discussions on K\"ahler manifolds and Anand Dessai for comments on an earlier version of this paper.

\section{(Generalized) complete intersections}
\label{intro}

In this section we define (generalized) complete intersections and state some of their properties.

We first define complete intersections in complex manifolds.

Let \(M\) be a compact complex manifold and \(L_1,\dots,L_k\) be base point free holomorphic line bundles over \(M\).
In this case, by Bertini's theorem \cite[p. 137]{MR1288523}, a generic holomorphic section \(s\) of \(V=L_1\oplus\dots\oplus L_k\) is transversal to all subbundles of the form \(L_{i_1}\oplus \dots \oplus L_{i_{k'}}\) with \(1\leq i_j< i_{k'}\leq k\), \(0\leq k'<k\).
Let \(X=s^{-1}(0)\) be the zero locus of such a section.
Then \(X\) is a complex submanifold of \(M\) which is the transversal intersection of complex hypersurfaces \(X_1,\dots,X_k\subset M\).
We call such an \(X\) a complete intersection in \(M\).
As noted by Thom the diffeomorphism type of \(X\) only depends on the isomorphism classes of the holomorphic line bundles \(L_1,\dots,L_k\).

We now turn to the definition of generalized complete intersections in oriented manifolds (cf. \cite[Chapter 3]{MR1189136}).

Let \(M\) be a closed oriented manifold and \(L_1,\dots,L_k\) be complex line bundles over \(M\).
In this case, a generic smooth section \(s\) of \(V=L_1\oplus\dots\oplus L_k\) is transversal to all subbundles of the form \(L_{i_1}\oplus \dots \oplus L_{i_{k'}}\) with \(1\leq i_j< i_{k'}\leq k\), \(0\leq k'<k\).
Let \(X=s^{-1}(0)\) be the zero locus of such a section.
Then \(X\) is a closed oriented submanifold of \(M\) which is the transversal intersection of closed oriented submanifolds \(X_1,\dots,X_k\subset M\) of codimension two.
We call such an \(X\) a generalized complete intersection in \(M\).
In this case the rational oriented bordism class of \(X\) only depends on the isomorphism classes of the complex line bundles \(L_1,\dots,L_k\) \cite[p. 35]{MR1189136}.

In both cases the normal bundle of \(X\) in \(M\) is given by the restriction of \(V\) to \(X\).
Moreover, the Poincare duals of the \(X_i\) are given by the first Chern classes of the \(L_i\).
In particular, the Poincare dual of \(X\) is given by the Euler class of \(V\).

\section{$\text{Spin}^c$-manifolds}
\label{sec:textspinc-manifolds}

In this section we review some known facts about \(\text{Spin}^c\)-manifolds.
For more background information about the group \(\text{Spin}^c(k)\) and \(\text{Spin}^c\)-structures on manifolds see for example \cite{0146.19001}, \cite{0247.57010} and \cite{0395.57020}.

An orientable manifold \(M\) has a \(\text{Spin}^c\)-structure if and only if the second Stiefel-Whitney-class \(w_2(M)\)  of \(M\) is the \(\!\!\!\!\mod 2\)-reduction of an integral class \(c\in H^2(M;\mathbb{Z})\).
Associated to a \(\text{Spin}^c\)-structure on \(M\) there is a complex line bundle.
We denote the first Chern-class of this line bundle by \(c_1^c(M)\).
Its \(\!\!\!\!\mod 2\)-reduction is \(w_2(M)\) and we should note that any integral cohomology class whose \(\!\!\!\!\mod 2\)-reduction is \(w_2(M)\) may be realized as the first Chern-class of a line bundle associated to some \(\text{Spin}^c\)-structure on \(M\).

Let \(M\) be a closed \(2n\)-dimensional \(\text{Spin}^c\)-manifold, i.e. a closed \(2n\)-dimensional manifold together with a choice of a \(\text{Spin}^c\)-structure  on it.
We then have a \(\text{Spin}^c\)-Dirac operator \(\partial_c\) on \(M\).
Its index is defined as
\begin{equation*}
  \ind(\partial_c) = \dim \ker \partial_c - \dim \coker \partial_c \in \mathbb{Z}.
\end{equation*}

Let \(V\) be a complex vector bundle over \(M\) and \(W\) an even-dimensional \(\text{Spin}\) vector bundle over \(M\).
With this data we build a power series \(R\in K(M)[[q]]\) defined by
\begin{align*}
  R&= \bigotimes_{k=1}^\infty S_{q^k}(\tilde{TM}\otimes_\R \C)\otimes \Lambda_{-1}(V^*)\otimes \bigotimes_{k=1}^\infty \Lambda_{-q^k}(\tilde{V}\otimes_\R \C)\\& \otimes \Delta(W)\otimes\bigotimes_{k=1}^\infty \Lambda_{q^n}(\tilde{W}\otimes_\R\C).
\end{align*}
Here \(q\) is a formal variable, \(\tilde{E}\) denotes the reduced vector bundle \(E -\dim E\), \(\Delta(W)\) is the full complex spinor bundle associated to the \(\text{Spin}\)-vector bundle \(W\), and \(\Lambda_t\) (resp. \(S_t\)) denotes the exterior (resp. symmetric) power operation. The tensor products are, if not indicated otherwise, taken over the complex numbers.

After extending \(\ind\) to power series we may define (cf. \cite[Definition 2.1]{0963.19002} and \cite[Definition 3.1]{0944.58019}):

\begin{definition}
  Let \(\varphi^c(M;V,W)\) be the index of the \(\text{Spin}^c\)-Dirac operator twisted with \(R\):
  \begin{equation*}
    \varphi^c(M;V,W)= \ind(\partial_c \otimes R)\in \mathbb{Z}[[q]].
  \end{equation*}
\end{definition}

The Atiyah-Singer index theorem \cite{0164.24301} allows us to calculate 
\begin{equation}
  \label{eq:2}
  \varphi^c(M;V,W)=\langle e^{c_1^c(M)/2}\ch(R)\hat{A}(TM),[M]\rangle.
\end{equation}
Here we have
\begin{equation}
  \label{eq:3}
  \ch(R)=Q_1(TM)Q_2(V)Q_3(W)
\end{equation}
with
\begin{align*}
  Q_1(E)&=\ch(\bigotimes_{k=1}^{\infty} S_{q^k}(\tilde{E}\otimes_\R \C))=\prod_i\prod_{k=1}^\infty \frac{(1-q^k)^2}{(1-e^{x_i}q^k)(1-e^{-x_i}q^k)},\\
  Q_2(V)&=\ch( \Lambda_{-1}(V^*)\otimes \bigotimes_{k=1}^\infty \Lambda_{-q^k}(\tilde{V}\otimes_\R \C))\\ &= \prod_i (1-e^{-v_i})\prod_{k=1}^{\infty} \frac{(1-e^{v_i}q^k)(1-e^{-v_i}q^k)}{(1-q^k)^2},\\
  Q_3(W)&=\ch(\Delta(W)\otimes\bigotimes_{k=1}^\infty \Lambda_{q^n}(\tilde{W}\otimes_\R\C))\\ &=\prod_i(e^{w_i/2}+e^{-w_i/2})\prod_{k=1}^{\infty} \frac{(1+e^{w_i}q^k)(1+e^{-w_i}q^k)}{(1+q^k)^2},
\end{align*}
and
\[\hat{A}(E)=\prod_{i}\frac{x_i}{e^{x_i/2}-e^{-x_i/2}}.\]
Here \(\pm x_i\) (resp. \(v_i\) and \(\pm w_i\)) denote the formal roots of the vector bundle \(E\) (resp. \(V\) and \(W\)).
If \(c_1^c(M)=c_1(V)\) modulo torsion, then we have
\begin{equation}
  \label{eq:4}
  e^{c_1^c(M)/2}Q_2(V)= \frac{e(V)}{\hat{A}(V)}\prod_i\prod_{k=1}^{\infty} \frac{(1-e^{v_i}q^k)(1-e^{-v_i}q^k)}{(1-q^k)^2}=\frac{e(V)}{Q_1(V)\hat{A}(V)}.
\end{equation}

Note that if \(M\) is a \(\text{Spin}\)-manifold,
then there is a canonical \(\text{Spin}^c\)-structure on \(M\) (with \(c_1^c(M)=0\)).
With this \(\text{Spin}^c\)-structure  the elliptic genus of \(M\) is equal to \(\varphi^c(M;0,TM)\).
Moreover, \(\varphi^c(M;0,0)\) is the Witten-genus of \(M\).

Let \(G\) be a compact Lie-group such that:
\begin{enumerate}
\item \label{item:1} There is an exact sequence of Lie-groups
  \begin{equation*}
    1 \rightarrow T \rightarrow G\rightarrow W(G)\rightarrow 1,
  \end{equation*}
where \(T\) is a torus and \(W(G)\) a finite group.
\item\label{item:2} If condition (\ref{item:1}) holds, then \(G\) acts by conjugation on \(T\).
Since \(T\) is abelian this action factors through \(W(G)\). We assume that this action of \(W(G)\) is non-trivial on \(T\).
\end{enumerate}

An action of \(G\) on a manifold \(M\) is called nice if \(G\) acts almost effectively on \(M\) and if the induced action on \(H^*(M;\mathbb{Z})\) is trivial.

For nice \(G\)-actions on \(\text{Spin}^c\)-manifolds we proved the following vanishing result as a generalization of \cite[Theorem 4.4]{0963.19002}.

\begin{theorem}[{\cite[Theorem 2.4]{MR3031643}}]
  \label{sec:introduction}
  Let \(M\) be a closed connected \(\text{Spin}^c\)-manifold on which a compact Lie group \(G\) satisfying the above conditions acts nicely such that \(M^G\neq \emptyset\).
  Let \(V\) and \(W\) be sums of complex line bundles over \(M\) such that \(W\) is \(\text{Spin}\), \(c_1(V)=c_1^c(M)\) modulo torsion and \(p_1(V+W-TM)=0\) modulo torsion. 
Assume that \(b_1(M)=0\) or that the \(G\)-action on \(M\) extends to an action of a simply connected compact Lie-group. 
Then \(\varphi^c(M;V,W)\) vanishes.
\end{theorem}

\section{Proofs of the main results}
\label{sec:proofs-main-results-4}

In this section we prove our main results.
We start with a lemma.

\begin{lemma}
  \label{sec:introduction-1}
  Let \(M\) be a closed \(\text{Spin}^c\)-manifold, \(X\subset M\) a generalized complete intersection, \(V\rightarrow M\) the complex vector bundle from Section~\ref{intro}.
  \begin{enumerate}
  \item\label{item:3} Then there is a \(\text{Spin}^c\) structure on \(X\) with \(c_1^c(X)=c_1^c(M)|_X-c_1(V)|_X\).
  \item\label{item:4}   Assume \(c_1^c(M)=c_1(V)\) modulo torsion and equip \(X\) with the \(\text{Spin}^c\) structure from (\ref{item:3}).
  Then we have
  \[\varphi^c(X;0,0)=\varphi^c(M;V,0).\]
  \end{enumerate}
\end{lemma}

\begin{proof}
  \emph{Ad (\ref{item:3}).} This follows from the discussion in the previous section because
  \[w_2(X)=w_2(M)|_X-w_2(V)|_X\equiv c_1^c(M)|_X-c_1(V)|_X\mod 2.\]
  
  \emph{Ad (\ref{item:4}).}   This is a simple calculation. With the notation from the previous section we have:
  \begin{align*}
    \varphi^c(X;0,0)&=\langle Q_1(TX)\hat{A}(TX),[X]\rangle\\
    &=\langle \frac{Q_1(TM)\hat{A}(TM)}{Q_1(V)\hat{A}(V)},[X]\rangle\\
    &=\langle e(V)\frac{Q_1(TM)\hat{A}(TM)}{Q_1(V)\hat{A}(V)},[M]\rangle\\
    &=\varphi^c(M;V,0).
  \end{align*}
  
  Here in the first line we have used equations (\ref{eq:2}), (\ref{eq:3}) and the fact that \(c_1^c(X)=c_1^c(M)|_X-c_1(V)|_X\) is torsion.
  In the second line we used that \(TM|_X=TX\oplus V|_X\) and
that \(Q_1(E\oplus E')=Q_1(E)Q_1(E')\) for vector bundles \(E\) and \(E'\) and similarly for \(\hat{A}\) instead of \(Q_1\).
In the third line we used that the Euler class of \(V\) is the Poincare dual of \(X\).
In the last line we used equations (\ref{eq:2}), (\ref{eq:3}) and (\ref{eq:4}) together with the assumption that \(c_1(V)=c_1^c(M)\) modulo torsion.
\end{proof}

The following theorem is Theorem~\ref{sec:introduction-6} from the introduction.

\begin{theorem}
  \label{sec:proofs-main-results}
  Let \(G\) be a compact simply connected Lie group and \(H\subset G\) a closed subgroup which is not a maximal torus of \(G\) such that \(M=G/H\) is \(\text{Spin}^c\) and \(G\) acts almost effectively on \(M\).

  Let \(X\subset M\) be a generalized complete intersection and \(V\rightarrow M\) the bundle from Section~\ref{intro}.
  Assume that the following holds:
  \begin{enumerate}
  \item \(w_2(V)=w_2(M)\) and
  \item \(p_1(V)=p_1(M)\) modulo torsion.
  \end{enumerate}
  Then the Witten genus of \(X\) vanishes.
\end{theorem}
\begin{proof}
  By the condition on \(w_2(V)\), we can fix a \(\text{Spin}^c\) structure on \(M\) such that \(c_1^c(M)=c_1(V)\).
  By Lemma~\ref{sec:introduction-1}, we then have to show that the index \(\varphi^c(M;V,0)\) vanishes.

  When \(H\) does not have maximal rank in \(G\), there is an \(S^1\subset G\) which has no fixed points in \(M\).
  By Corollary 1.2 in \cite{0346.57014} and Lemmas 2.1 and 2.7 in \cite{MR3031643}, the action of this \(S^1\) lifts into the \(\text{Spin}^c\)-structure on \(M\) and the vector bundle \(V\) over \(M\).
  Hence there is an \(S^1\)-equivariant refinement of the index \(\varphi^c(M;V,0)\) which vanishes by the Lefschetz-fixed point formula.
  Therefore the claim follows in this case.

  If the identity component of \(H\) is a maximal torus of \(G\), we can argue as follows.
  Since the centralizer of a maximal torus of \(G\) is the torus itself, it follows that \(H\) is non-abelian and satisfies the assumptions of Theorem~\ref{sec:introduction}.
  Hence the claim follows in this case.

  If the identity component of \(H\) is non-abelian and has maximal rank in \(G\), then we can restrict the \(H\)-action on \(M\) to the normalizer \(N_H(T)\) of a maximal torus \(T\)  in \(H\) and argue as above.
\end{proof}

\begin{remark}
  \label{sec:proofs-main-results-1}
  \begin{itemize}
  \item We do not know whether our conditions on the characteristic classes of \(V\) are always satisfied if \(X\) is string. However they are satisfied in all cases where string (generalized) complete intersections in homogeneous spaces have been studied before.
  \item The situation where \(G/H\) is a irreducible compact Hermitian symmetric space has been studied before by Förster \cite{foerster07:_stolz}. He shows in his Lemmas 4.7, 5.17, 6.7, 8.4, 9.4 that our conditions of the theorem are satisfied in the case that \(X\) is string.
    Our theorem also deals with the cases which were incomplete in Förster's study.
  \item
Corollary 3.4 of \cite{foerster07:_stolz} deals with generalized complete intersections in products of complex projective spaces.
    These results have also been obtained by Chen-Han \cite{MR2373447} and can also be obtained from our result.
  \item The situation where \(G/H\) is a product of complex Grassmannians has been studied by Zhou-Zhuang \cite{MR3275823}. By their Proposition 5.2, their result is also a special case of our result.
  \item The situation where \(G/H\) is a complex flag manifold was studied by Zhuang \cite{MR3546614}. By his Propositions 4.1, 4.2, his result in the string case is also a special case of our result. However he also studied generalized Witten genera for string$^c$ structures on the complete intersections in this case. We do not generalize these results.
  \end{itemize}
\end{remark}

\begin{remark}
\label{sec:proofs-main-results-3}
  The theorem is not true, when \(H\) is a maximal torus of \(G\).
  Indeed, then the tangent bundle of \(M\) splits as a sum of complex line bundles and the Euler-characteristic \(\chi(M)\) is positive.
  Hence one can realize a set of \(\chi(M)\) isolated points in \(M\) as a generalized complete intersection in \(M\) with \(V=TM\).
  But of course the Witten genus of this set does not vanish.
  We do not know whether the theorem does or does not hold, when one adds the extra condition that \(X\) has positive dimension.
\end{remark}

Similarly to Theorem~\ref{sec:proofs-main-results} one can prove the following theorem which is Theorem~\ref{sec:introduction-5} from the introduction.

\begin{theorem}
  \label{sec:introduction-4}
  Let \(M\) be a closed connected \(Spin^c\) manifold such that
  \begin{enumerate}
  \item There is an almost effective action of a simply connected compact Lie group \(G\) on \(M\) such that \(M^{G}\neq \emptyset\), or
  \item \(b_1(M)=0\) and there is an almost effective action of a torus \(T\) on \(M\) with \(\dim T>b_2(M)\).
  \end{enumerate}

  Let, moreover, \(X\subset M\) be a generalized complete intersection and \(V\rightarrow M\) the bundle from Section~\ref{intro}.
  Assume that the following holds:
  \begin{enumerate}
  \item \(w_2(V)=w_2(M)\) and
  \item \(p_1(V)=p_1(M)\) modulo torsion.
  \end{enumerate}
  Then the Witten genus of \(X\) vanishes.
\end{theorem}
\begin{proof}
  As in the proof of Theorem~\ref{sec:proofs-main-results}, we fix a \(\text{Spin}^c\)-structure on \(M\) such that \(c_1^c(M)=c_1(V)\).
  Then, by Lemma~\ref{sec:introduction-1}, it suffices to show that \(\varphi^c(M;V,0)\) vanishes.

  If we have an action of a simply connected compact Lie group on \(M\), then we can restrict this action to the normalizer of a maximal torus in \(G\) and deduce the result from Theorem \ref{sec:introduction}.

  If we only have an action of a torus on \(M\), then we can argue as in the proof of Theorem 3.3 of \cite{MR3599868} to deduce the result.
  We omit the details.
\end{proof}

\section{Consequences}
\label{sec:consequences-1}

In this section we discuss applications of our Theorems \ref{sec:proofs-main-results} and \ref{sec:introduction-4}.

As a first consequence of Theorem~\ref{sec:proofs-main-results} we can complete the proof of Förster's vanishing result for Witten genera of string complete intersections in irreducible compact Hermitian symmetric spaces.
The result is as follows.

\begin{cor}
  Let \(M\) be a irreducible compact Hermitian symmetric space and \(X\subset M\) a string complete intersection.
  Then the Witten genus of \(X\) vanishes.
\end{cor}
\begin{proof}
  This follows from Theorem \ref{sec:proofs-main-results} and the lemmas from \cite{foerster07:_stolz} mentioned in Remark~\ref{sec:proofs-main-results-1}.
\end{proof}

To give a second application for our theorems, we have to recall the definition of (generalized) Bott manifolds.
This is done inductively:

A single point is the only \(0\)-stage (generalized) Bott manifold.
A \(k\)-stage generalized Bott manifold, \(k\geq 1\), is the total space of a \(\mathbb{C} P^{n_k}\)-bundle, \(n_k\geq 1\), with structure group a torus acting lienarly on the fiber over an \((k-1)\)-stage generalized Bott manifold.
A \(k\)-stage Bott manifold, \(k\geq 1\), is the total space of a \(\mathbb{C} P^1\)-bundle with structure group a torus acting linearly on the fiber over an \((k-1)\)-stage Bott manifold.

All (generalized) Bott manifolds are unitary, simply connected and examples of so-called torus manifolds.
A torus manifold \(M\) is a closed connected oriented manifold of real dimension \(2n\) with an effective action of a compact torus \(T\)  of dimension \(n\) such that \(M^T\neq \emptyset\).
Note that for a \(k\)-stage generalized Bott manifold \(M\) we have
\[b_2(M)=k\leq \sum_{i=1}^k n_i=\frac{1}{2}\dim M\]
with equality if and only if \(M\) is a Bott manifold.

So Theorem~\ref{sec:introduction-4}, in particular, shows the following (generalizing a result from \cite{han2023iterated}):

\begin{cor}
\label{sec:proofs-main-results-2}
Let \(M\) be
\begin{enumerate}
\item a generalized Bott manifold which is not a Bott manifold or more generally
\item  a \(\text{Spin}^c\) torus manifold with \(b_1(M)=0\) and \(b_2(M)<\frac{1}{2}\dim M\).
\end{enumerate}

    Let moreover \(X\subset M\) be a generalized complete intersection and \(V\rightarrow M\) the bundle from Section~\ref{intro}.
  Assume that the following holds:
  \begin{enumerate}
  \item \(w_2(V)=w_2(M)\) and
  \item \(p_1(V)=p_1(M)\) modulo torsion.
  \end{enumerate}
  Then the Witten genus of \(X\) vanishes.
\end{cor}

\begin{remark}
 The example \(M=\prod \mathbb{C} P^1=\prod SU(2)/S^1\) shows that the corollary is not true when \(M\) is a Bott manifold (see Remark~\ref{sec:proofs-main-results-3}).
 Note that more general examples are given in Lemma 5.3 of \cite{MR3031643}.
 These examples show that the corollary does not hold for any Bott manifold.
\end{remark}

 \begin{cor}
   \label{sec:consequences}
   Let \(M\) be as in Theorem
   \ref{sec:introduction-4} and \(X\subset M\) a generalized complete intersection which is string such that
   \begin{enumerate}
   \item \(X\hookrightarrow M\) is four-connected, or, more generally,
   \item \(H^2(M;\mathbb{Z}_2)\rightarrow H^2(X;\mathbb{Z}_2)\) and \(H^4(M;\mathbb{Q})\rightarrow H^4(X;\mathbb{Q})\) are injective.
   \end{enumerate}
   Then the Witten genus of \(X\) vanishes.
 \end{cor}
 \begin{proof}
   If \(X\) is string, then we have \(w_2(M)|_X-w_2(V)|_X=w_2(X)=0\) and \(p_1(M)|_X-p_1(V)|_X=p_1(X)=0\).
   Hence the corollary follows.
 \end{proof}

We now discuss two corollaries of Theorem~\ref{sec:introduction-4} which give new evidence for the Stolz conjecture.

 Let \(M\) be a Riemannian manifold.
 Positive \(k\)-th intermediate Ricci curvature for \(1\leq k\leq \dim M -1\) is a curvature condition which interpolates between positive sectional curvature and positive Ricci curvature.
 \(M\) satisfies this condition if for all \(x\in M\) and all orthonormal \(k+1\)-frames \( \{u,e_1,\dots,e_k\}\subset T_xM\) one has \(\sum_{i=1}^k K(u,e_i)>0\), where \(K(u,e_i)\) denotes the sectional curvature of the plane spanned by \(u\) and \(e_i\).

 So the case \(k=1\) is precisely the case of positive sectional curvature and \(k=\dim M-1\) is positive Ricci curvature.
 Note that if \(M\) has positive \(k\)-th intermediate Ricci curvature, it also has positive \(k'\)-th intermediate Ricci curvature for all \(k\leq k'\leq \dim M-1\).

 Together with Wilking's connectedness lemma \cite{MR2051400} for intermediate Ricci curvature, the above corollary implies the following:

 \begin{cor}
   Let \(M\) be as in Theorem~\ref{sec:introduction-4}.
   Assume that there is a Riemannian metric on \(M\) with positive \(k\)-th intermediate Ricci curvature.
   Let \(X\subset M\) be a generalized complete intersection such that all the codimension two submanifolds defining \(X\) are totally geodesic with respect to the above metric.
   If
   \begin{enumerate}
   \item \(1\leq k\leq \min \{\dim M - 6, \dim X -3\}\) and
   \item \(X\) is string,
   \end{enumerate}
   then the induced metric on \(X\) has positive Ricci curvature and the Witten genus of \(X\) vanishes.
 \end{cor}
 \begin{proof}
   The first claim follows because \(X\) as a totally geodesic submanifold of \(M\) has positive \(k\)-th intermediate Ricci curvature and \(k<\dim X-1\).

   To prove the second claim let \(X_1,\dots,X_l\) be the codimension two submanifolds defining \(X\).
   It follows from Remark 2.4 in \cite{MR2051400} that the inclusions
   \begin{align*}
     X_1&\hookrightarrow M, & \bigcap_{i=1}^{j+1}X_i&\hookrightarrow \bigcap_{i=1}^{j}X_i,
   \end{align*}
for \(j=1,\dots,l-1\) are four-connected.
   Hence the second claim is implied by Corollary~\ref{sec:consequences}.
 \end{proof}

 Our other corollary is for Fano manifolds.
We need the following two lemmas to prove it.
 For a K\"ahler manifold \(M\), we denote by \(h^{p,q}(M)\) the Hodge numbers of \(M\). 
 
\begin{lemma}
  \label{sec:compl-inters-kahl}
  Let \(M\) be a simply connected compact Kähler manifold with \(h^{1,1}(M)=1\) and \(X\subset M\) a complete intersection of complex dimension at least \(4\).
  Then the inclusion \(X\hookrightarrow M\) is four-connected or \(X\) is empty.
\end{lemma}
\begin{proof}
  Let \[L_1,\dots, L_k\] be the holomorphic line bundles defining \(X\) and \[X_1,\dots,X_k\] the complex hypersurfaces of \(M\) such that \(X=\bigcap_{i=1}^k X_i\).

  If one of the first Chern classes of the \(L_i\) is zero, then \(X_i\) is empty. Therefore \(X=\emptyset\) in this case.
  
  Hence we can assume that all \(c_1(L_i)\) are non-zero.
  Since \(h^{1,1}(M)=1\) and \(M\) is Kähler, each \(L_i\) must be either negative or positive.
  Since negative line bundles do not have non-zero holomorphic sections, all the \(L_i\) are positive.

  Hence, it follows from the Lefschetz hyperplane theorem (see \cite{MR0215323}) and the assumption \(\dim_{\mathbb{C}} X\geq 4\) that all the inclusions
  \[\bigcap_{i=1}^{j+1} X_i\hookrightarrow \bigcap_{i=1}^{j} X_i,\]
  \(j=0,\dots,k-1\), are four-connected.

  Hence the claim follows.
\end{proof}

\begin{lemma}
  \label{sec:consequences-2}
  Let \(M\) be a Fano manifold with \(b_2(M)=1\) and \(\dim_\mathbb{C} M=n\) and \(X\subset M\) a string complete intersection with \(\dim_{\mathbb{C}} X\geq 4\).
  Then \(X\) is Fano.
\end{lemma}
\begin{proof}
  Let \(u\in H^2(M;\mathbb{Z})\) be the positive generator, \(l_1,\dots,l_k\in \mathbb{Z}\) and \(m\in \mathbb{Z}\) such that
  \begin{align*}
    c_1(L_i)&=l_i u,&c_1(M)&=mu.
  \end{align*}
  Here the \(L_i\) are the line bundles which define \(X\).
  As shown in the proof of Lemma~\ref{sec:compl-inters-kahl}, we can assume \(m,l_i>0\) for all \(i\).

  We have to show that \(m-\sum_{i=1}^k l_i\) is positive.
  Since \(X\hookrightarrow M\) is four-connected by Lemma~\ref{sec:compl-inters-kahl} and \(0=p_1(X)=p_1(M)|_X-p_1(V)|_X\), we have
  \[m^2u^2-2c_2(M)=c_1(M)^2-2c_2(M)=p_1(M)=p_1(V)=\sum_{i=1}^k l_i^2 u^2\]
  and, therefore,
  \[2c_2(M)=(m^2-\sum_{i=1}^k l_i^2) u^2.\]

  By the main result of \cite{MR4000253} we have
  \[\langle c_2(M)u^{n-2},[M]\rangle\geq \frac{m(m-1)}{2}\langle u^n,[M]\rangle,\]
  with equality only if \(m\geq n-1\) (for this extension see the last paragraph of the proof in \cite{MR4000253}).
  Since \(\langle u^n,[M]\rangle>0\), it follows that \[m\geq \sum_{i=1}^k l_i^2\geq \sum_{i=1}^k l_i.\]
  Here in the first inequality we have strict inequality whenever \(m< n-1\).
  Moreover in the second inequality we have equality if and only if all \(l_i\) are equal to one.
  Hence we have \(m>\sum_{i=1}^k l_i\) unless \(k=m\geq n-1\).
  But this exceptional case is excluded by our assumption on the dimension of \(X\).
  So the lemma is proved.
\end{proof}

By a combination of Lemmas~\ref{sec:compl-inters-kahl}, \ref{sec:consequences-2} and Corollary~\ref{sec:consequences} we get the following corollary which is Corollary~\ref{sec:introduction-7} in the introduction.

\begin{cor}
  \label{sec:consequences-3}
    Let \(M\) be a Fano manifold with \(b_2(M)=1\). Assume that there is a smooth effective action of a torus \(T\) with \(\dim T=2\) on \(M\) or that there is a smooth almost effective action of a simply connected compact Lie group on \(M\) with a fixed point.
  
  Let, moreover, \(X\subset M\) be a string complete intersection of complex dimension at least four.

  Then \(X\) is Fano and the Witten genus of \(X\) vanishes.
\end{cor}

We say that a \(2n\)-dimensional closed manifold \(M\) satisfies the Hard Lefschetz property if there is a \(u\in H^2(M;\mathbb{R})\) such that for all \(k=0,\dots,n\) the cup product with \(u^{n-k}\) induces an isomorphism
\[H^k(M;\mathbb{R})\rightarrow H^{2n-k}(M;\mathbb{R}).\]

By the Hard Lefschetz theorem \cite[p. 122]{MR1288523} compact Kähler manifolds have the Hard Lefschetz property.
In symplectic geometry there is a conjecture that a closed symplectic manifold which admits an Hamiltonian torus action with isolated fixed points satisfies the Hard Lefschetz property (see \cite[Section 4.2]{birs05}).
So the next corollary might also be of some interest.

\begin{cor}
  Let \(M\) be a closed simply connected \(\text{Spin}^c\)-manifold satisfying the Hard Lefschetz property with \(H^2(M;\mathbb{Z})=\mathbb{Z}\) such that there is a smooth effective action of a torus \(T\) with \(\dim T=2\) on \(M\).
  Let \(u\) be a generator of \(H^2(M;\mathbb{Z})\) and assume
  \begin{equation}\label{eq:1}c_1^c(M)^2-p_1(M)\in 2\mathbb{Z}u^2\subset H^4(M;\mathbb{R}).\end{equation}
  
  Let, moreover, \(X\subset M\) be a string generalized complete intersection of real dimension at least eight.

  Then the Witten genus of \(X\) vanishes.
\end{cor}
\begin{proof}
  Let \(V=L_1\oplus\dots\oplus L_k\) be the vector bundle over \(M\) from section~\ref{intro} with complex line bundles \(L_i\).
  By Theorem~\ref{sec:introduction-4}, it suffices to show that \(p_1(V)=p_1(M)\) modulo torsion and \(w_2(V)=w_2(M)\).

  By using the Hard Lefschetz property and the dimension and string assumptions on \(X\) one can see as in the proof of Corollary 2.36 in \cite{foerster07:_stolz} that \(p_1(V)=p_1(M)\in H^4(M;\mathbb{R})\).

  To see the equation for \(w_2(M)\) and \(w_2(V)\) let \(m,l_i\in \mathbb{Z}\) such that
  \begin{align*}
    c_1^c(M)&=mu,&c_1(L_i)=l_iu.
  \end{align*}
  Then by our assumption (\ref{eq:1}) we have
  \[(m^2-\sum_{i=1}^k l_i^2)u^2=c_1^c(M)^2-p_1(V)=c_1^c(M)^2-p_1(M)\in 2\mathbb{Z}u^2\subset H^4(M;\mathbb{R}).\]
  Since, by the Hard Lefschetz property, \(u^2\) is not a torsion class, it follows that \(m\equiv m^2\equiv \sum_{i=1}^k l_i^2\equiv\sum_{i=1}^k l_i \mod 2\).
  Therefore we get the claim about the second Stiefel-Whitney classes.
\end{proof}

\begin{remark}
  \begin{itemize}
  \item   If in the situation of the above corollary the \(\text{Spin}^c\)-structure on \(M\) is induced from a unitary structure then \(c_1^c(M)^2-p_1(M)=2c_2(M)\).
  \item As shown in \cite{foerster07:_stolz} the condition~(\ref{eq:1}) is satisfied in all but one case when \(M\) is an irreducible compact Hermitian symmetric space.
  \end{itemize}
\end{remark}

\bibliography{dirac}{}
\bibliographystyle{alpha}
\end{document}